\documentclass[a4paper,12pt]{amsart}
\usepackage{latexsym,amscd,amssymb,amsmath,mathrsfs,comment,url}
\makeatletter
\theoremstyle{plain}
  \newtheorem{thm}{Theorem}[section]
  \newtheorem{prop}[thm]{Proposition}
  \newtheorem{lem}[thm]{Lemma}
  \newtheorem{cor}[thm]{Corollary}
  \newtheorem{conj}[thm]{Conjecture}
\theoremstyle{definition}
  \newtheorem{dfn}[thm]{Definition}
  \newtheorem{exmp}[thm]{Example}

  \newtheorem{rem}[thm]{Remark}
\def\theDcat{{\mathsf D}}
\def\Db{\theDcat^b} 

\let\opn\operatorname 
\newcommand\@bothmode[1]{\ifmmode #1\else $#1$ \fi}
\numberwithin{equation}{section}
\renewcommand\theenumi{\@roman\c@enumi}
\renewcommand{\@biblabel}[1]{#1.}
\newcommand\chdir[1]{%
\let\pwd\relax%
\def\pwd{#1}}
\newcommand\inputs[1]{\@for\@tempmember:=#1\do{\input{\pwd /\@tempmember}}}
\DeclareFontFamily{U}{futm}{}
\DeclareFontShape{U}{futm}{m}{n}{
  <-> s * [.92] fourier-bb
  }{}
\DeclareMathAlphabet{\math@bb}{U}{futm}{m}{n}
\let\mathbb\math@bb\relax
\DeclareMathAlphabet{\math@cal}{OMS}{cmsy}{m}{n}
\let\mathcal\math@cal\relax
\newcommand\NN{\mathbb{N}} 
\newcommand\ZZ{\mathbb{Z}} 
\newcommand\kk{\mathbb{k}} 
\newcommand\ba{\mathbf{a}}
\newcommand\bb{\mathbf{b}}
\newcommand\bc{\mathbf{c}}

\newcommand\be{\mathbf{e}}

\newcommand\one{\mathbf{1}}

\let\@tempar\relax 
\def\@seton^#1{\overset{#1}{\@tempar}}
\def\@setbeneath_#1{\underset{#1}{\@tempar}}
\newcommand\defar[2]{\@xp\def\csname #1\endcsname{
    \def\@tempar{#2}\@ifnextchar^{\@seton}{
        \@ifnextchar_{\@setbeneath}{\@tempar}
    }}}
\defar{longto}{\longrightarrow} 
\defar{epito}{\twoheadrightarrow} 
\defar{monoto}{\rightarrowtail} 
\defar{embto}{\hookrightarrow} 
\newcommand\imply{\@bothmode{\Rightarrow}} 
\newcommand\Imply{\@bothmode{\Longrightarrow}} 
\renewcommand\iff{\@bothmode\Leftrightarrow} 
\newcommand\Iff{\stackrel{\rm def}{\Longleftrightarrow}}
\newcommand\get{\@bothmode{\Leftarrow}} 
\newcommand\Get{\@bothmode{\Longleftarrow}} 
\newcommand\bra[1]{[#1]} 
\newcommand\mbra[1]{\left\{ #1 \right\}} 
\newcommand\set[2]{\mbra{\,#1\ \left|\ #2\, \right.}} 

\newcommand\defopn[1]{\@xp\def\csname #1\endcsname{\opn{\csname the#1\endcsname\relax}}}
\defopn{Ker}

\defopn{Im} 
\newcommand\ann{\opn{ann}}
\newcommand\depth{\opn{depth}_S}
\newcommand\sdepth{\opn{sdepth}_S}

\newcommand\sd{\opn{sd}}

\newcommand\pol{\opn{\mathsf{pol}_\ba}}
\newcommand\dpol{\opn{\mathsf{pol}^\ba}}
\newcommand\Pol{\opn{\mathsf{pol}_{\ba+\one}}}
\newcommand\dPol{\opn{\mathsf{pol}^{\ba+\one}}}
\newcommand\adeg{\opn{adeg}}
\newcommand\hdeg{\opn{hdeg}}
\newcommand\infl{\opn{\mathsf{Infl}}}
\newcommand\BM{\opn{\mathsf{BM}}}
\newcommand\BBM{\opn{\mathcal{B \! M}}}
\newcommand\seq[2][n]{{#2}_1,\dots ,{#2}_{#1}}
\newcommand\poly[2][\kk]{#1\bra{#2}}
\def\@polys[#1][#2]#3{\poly[#1]{\seq[#2]#3}}
\def\@polysx[#1]{\@ifnextchar[{\@polys[#1]}{\@polys[\kk][#1]}}
\def\polys{\@ifnextchar[{\@polysx}{\@polys[\kk][n]}}
\newcommand\fp{\mathfrak{p}}
\def\<{{\langle}}
\def\>{{\rangle}}
\def\tl{\triangleleft}
\def\tr{\triangleright}
\renewcommand\mod{\opn{mod}}
\def\0dt{\opn{mod}_{{\bf 0.dt}}}

\newcommand\Sq{\opn{Sq}}
\newcommand\wS{\widetilde{S}}
\newcommand\hS{\widehat{S}}
\newcommand\Ass{\opn{Ass}}
\newcommand\Hom{\opn{Hom}}
\newcommand\Tor{\opn{Tor}}

\newcommand\Ext{\opn{Ext}}

\newcommand\sA{\mathscr{A}}
\newcommand\DD{\mathbb{D}}

\newcommand\op{\mathsf{op}}

\newcommand\cD{\mathcal{D}}
\makeatother
\title[Sliding functor and polarization functor]
{Sliding functor and polarization functor for multigraded modules}
\author{Kohji Yanagawa}
\thanks{The author is partially supported by Grant-in-Aid for Scientific Research (c) (no.19540028).}
\address{Department of Mathematics, Kansai University,
Suita 564-8680, Japan}
\email{yanagawa@ipcku.kansai-u.ac.jp}

\begin{document}

\maketitle

\begin{abstract}
We define {\it sliding functors}, which are exact endofunctors of the category of 
multi-graded modules over a polynomial ring. 
They preserve several invariants of modules, especially  the 
(usual) depth and Stanley depth. In a similar way, we can also define  
the {\it polarization functor}. While this idea has appeared in 
papers of Bruns-Herzog and Sbarra, we give slightly different approach.    
Keeping these functors in mind, we treat simplicial spheres of Bier-Murai type.  
\end{abstract}

\section{Introduction} 
Let $S:=\kk[x_1, \ldots, x_n]$ be a polynomial ring over a field $\kk$, and let 
$\mod_{\ZZ^n} S$ denote the category of $\ZZ^n$-graded finitely generated $S$-modules. 

For $i \in [n]:=\{1, \ldots, n\}$ and $j \in \ZZ$, we define the {\it sliding functors} 
$(-)^{\tl\<i,j\>}$ and $(-)^{\tr\<i,j\>}$, 
which are exact endofunctors of $\mod_{\ZZ^n} S$ with 
$S(-\ba)^{\tl\<i,j\>} = S(-\ba')$ and $S(-\ba)^{\tr\<i,j\>} =S(-\ba'')$. 
Here, $a_k=a'_k=a''_k$ for all $k \ne i$ and 
$$
a'_i=
\begin{cases}
a_i + 1 & \text{if $a_i \geq j$,}\\
a_i & \text{if $a_i < j$,}
\end{cases}
\quad \text{and} \quad 
a''_i:=
\begin{cases}
a_i & \text{if $a_i > j$,} \\
a_i - 1 & \text{if $a_i \leq j$} 
\end{cases}
$$ 
(see the beginning of the next section for the convention on the vector $\ba \in \ZZ^n$). 
If $I$ is a monomial ideal and $j >0$, then $I^{\tl \<i,j\>}$ is a monomial ideal 
with $(S/I)^{\tl \<i,j\>}=S/I^{\tl\<i,j\>}$.  

Sliding functors preserve the depth and dimension.  
Moreover, we have  
$$\Ext_S^l(M^{\tl  \<i,j\>}, S) \cong 
\Ext_S^l(M, S)^{\tr  \<i,-j\>} \quad \text{and} \quad 
\Ext_S^l(M^{\tr  \<i,j\>}, S) \cong 
\Ext_S^l(M, S)^{\tl  \<i,-j\>}$$ 
for $M \in \mod_{\ZZ^n} S$.  
Hence these functors also preserve 
the sequentially Cohen-Macaulay property and Serre's condition $(S_r)$. 

The {\it Stanley depth} of $M \in \mod_{\ZZ^n} S$, which is denoted by $\sdepth M$, 
is a combinatorial invariant. 
It is conjectured that the inequality $\sdepth M \geq \depth M$ holds for all $M$ ({\it Stanley's conjecture}).  In Theorem~\ref{sdepth}, we will show that 
$$\sdepth M = \sdepth M^{\tl \<i,j\>}=\sdepth M^{\tr \<i,j\>}.$$
Hence Stanley's conjecture holds for $M$ if and only if it holds for $M^{\tl \<i,j\>}$ 
(or $M^{\tr \<i,j\>}$). 
It means that the conjecture holds for all monomial ideal if it holds for any 
monomial ideal $I$ satisfying the following condition:
\begin{itemize}
\item[$(*)$] Let $x^{\ba_1}, \ldots, x^{\ba_r}$ be the minimal generators of $I$.  
For each $i \in [n]$, there is a positive integer $b_i$ such that
$\{\, (\ba_l)_i \mid 1 \le l \le r \, \}$ is $\{ \, 1, 2, \ldots, b_i \, \}$  
or $\{ \, 0,1, \ldots, b_i \, \}$ 
(i.e., the degrees of minimal generators are ``consecutive").
\end{itemize}
Similarly, we can also reduce the problem to the case 
$\{\, (\ba_l)_i \mid  1 \le l \le r \, \}$ are arbitrary ``sparse".

Set $\one =(1, \ldots, 1) \in \NN^n$. For $\ba \in \NN^n$ with $\ba \succeq \one$, 
Miller (\cite{M}) 
introduced the notion of {\it positively $\ba$-determined $S$-modules}. 
They form the category $\mod_\ba S$, which is a full subcategory of $\mod_{\ZZ^n} S$. 
For a monomial ideal $I$ minimally generated by $x^{\bb_1}, \ldots, x^{\bb_r}$,  
we have $I \in \mod_\ba S$ if and only if $S/I \in \mod_\ba S$ 
if and only if $\bb_i \preceq \ba$ for all $i$.  
A positively $\one$-determined module is 
called a {\it squarefree module}. We denote the category of squarefree $S$-modules by $\Sq S$.   
A squarefree monomial ideal $I$ and its quotient ring $S/I$ (i.e., a Stanley-Reisner ring) 
are squarefree modules.  

Set $\wS:=\kk[\, x_{i,j} \mid 1 \leq i \leq n, 1 \leq j \leq a_i \,]$. 
{\it Polarization} is a classical technique constructing  
a squarefree monomial ideal $\pol(I) \subset \wS$ 
from a positively $\ba$ determined monomial ideal $I \subset S$. 
Extending this idea, we can define the functor $\pol: \mod_\ba S \to \Sq \wS$. 
The (essentially) same construction has already appeared in Bruns-Herzog \cite{BH} and 
Sbarra \cite{Sb}, but a new feature of the present paper is the ``reversed copy" 
$\dpol$ of $\pol$. Here 
$$\pol(x^\bb)=\prod_{i \in [n]} 
x_{i,1} x_{i,2} \cdots x_{i, b_i} \in \wS$$
and 
$$\dpol(x^\bb)=\prod_{i \in [n]} 
x_{i,a_i} x_{i,a_i-1} \cdots x_{i, a_i-b_i+1} \in \wS.$$  
(By abuse of notation, we identify a principal ideal with its generator in the above equations.)

We have 
$$\Ext_{\wS}^i(\pol(M), \wS(-\one)) \cong \dpol(\Ext_S^i(M,S(-\ba)))$$
for all $M \in \mod_\ba S$. 
This result corresponds to \cite[Corollary~4.10]{Sb}. 
However the pair $\pol$ and $\dpol$ is only implicit in \cite{Sb}, 
since a different convention is used there. This pair appears 
in a recent work of Murai \cite{Mu}, which we will discuss in \S5.  

Let $I \subset S$ be a squarefree monomial ideal 
(i.e., the Stanley-Reisner ideal of a simplicial complex $\Delta$), 
and set $S':= S[x_i']$ for some $i \in [n]$. 
In \cite{BH}, Bruns and Herzog considered the squarefree  monomial ideal $I' \subset S'$ given by 
replacing $x_i$ (appearing in the minimal generators of $I$) by $x_i \cdot x_i'$. 
In our context, $I'=\opn{\mathsf{pol}}_{\one+\be_i}(I^{\tl\<i,1\>})$. 
They called this operation (more precisely, the corresponding operation on $\Delta$) the 
{\it 1-vertex inflation}. 
In literature, the term ``an inflation of a simplicial complex" 
is sometime used in another meaning. So the reader should be careful.  

From a positively $\ba$-determined monomial ideal $I$ of $S$, 
Murai \cite{Mu} constructed the squarefree monomial ideal 
$\BM_\ba(I)$ of $\kk[\, x_{i,j} \mid 1 \leq i \leq n, 1 \leq j \leq a_i +1\,]$, which is the 
Stanley-Reisner ideal of  a simplicial sphere $\BBM_\ba(I)$ of dimension $-2 +\sum_{i=1}^n a_i$. 
If $I$ itself is squarefree, this construction is due to T. Bier (unpublished).   
Hence, we call simplicial complexes of the form $\BBM_\ba(I)$ {\it Bier-Murai spheres}. 
In Theorem~\ref{BM},  we show that 
the class of Bier-Murai spheres is closed under 1-vertex inflations.  

\section*{Acknowledgment.} The author is grateful to Professors Gunnar Fl\o ystad, 
Satoshi Murai and Naoki Terai 
for valuable information. 

\section{Sliding functors}
We introduce the convention on $\ZZ^n$ used throughout the paper. 
The $i^{\rm th}$ coordinate of $\ba \in \ZZ^n$ is denote by $a_i$ 
(i.e., we change the font). 
However, for a vector represented by several letters, such as $\ba_j \in \ZZ^n$, 
we denote its $i^{\rm th}$ coordinate by $(\ba_j)_i$. 
For $i \in [n]:=\{1, \ldots, n \}$, 
$\be_i$ denotes the $i^{\rm th}$ unit vector of $\ZZ^n$, that is, $(\be_i)_j=\delta_{i,j}$.  
Let $\succeq$ be the order on $\ZZ^n$ defined by 
$\ba \succeq \bb \ \Iff \ a_i \ge b_i$ for all $i$. Let $\ba \vee \bb$ denote 
the element of $\ZZ^n$ whose $i^{\text{th}}$-coordinate is $\max\mbra{a_i,b_i}$ for each $i$. 
Let $S:=\kk[x_1, \ldots, x_n]$ be a polynomial ring. 
For $\ba \in \NN^n$, $x^\ba$ denotes the monomial $\prod_{i \in [n]}x^{a_i}$ in $S$. 
Clearly, $S$ is a $\ZZ^n$-graded ring with $S=\bigoplus_{\ba \in \NN^n} \kk x^\ba$.

Let $\mod_{\ZZ^n} S$ be the category of finitely generated $\ZZ^n$-graded $S$-modules 
and their degree preserving $S$-homomorphisms. For $M \in \mod_{\ZZ^n} S$ and $\ba \in \ZZ^n$, 
$M(\ba)$ denotes the shifted module with $M(\ba)_\bb=M_{\ba+\bb}$, as usual. 
We say $M= \bigoplus_{\ba \in \ZZ^n} M_\ba$ is $\NN^n$-graded, if $M_\ba =0$ for all 
$\ba \not \in \NN^n$. Let $\mod_{\NN^n} S$ be the full subcategory of $\mod_{\ZZ^n} S$ 
consisting of $\NN^n$-graded modules. Recall that if $M,N \in \mod_{\ZZ^n} S$ then 
$\Ext_S^i(M,N) \in \mod_{\ZZ^n} S$ in the natural way.  As usual, for $M \in \mod_{\ZZ^n} S$, we call 
$\beta_{i,\ba}(M):=\dim_\kk [\Tor^S_i(\kk,N)]_\ba$ the $(i, \ba)^{\rm th}$ {\it Betti number} of $M$. 
We also treat ``coarser" Betti numbers   $\beta_{i,j}(M)$ and $\beta_i(M)$ 
(note that $M \in \mod_{\ZZ^n} S$ has the natural $\ZZ$-grading). 

From an order preserving map $q: \ZZ^n \to \ZZ^n$ and $M \in \mod_{\ZZ^n} S$, 
Brun and Fl\o ystad \cite{BF} constructed the new module 
$q^*M \in \mod_{\ZZ^n} S$ so that $(q^* M)_\ba = M_{q(\ba)}$ 
and the multiplication map $(q^*M)_{\ba} \ni y \longmapsto x^\bb y \in (q^*M)_{\ba+\bb}$ 
is given by $M_{q(\ba)} \ni y \longmapsto x^{q(\ba+\bb)-q(\ba)} y \in M_{q(\ba+\bb)}$ 
for all $\ba \in \ZZ^n$ and $\bb \in \NN^n$.  Similarly,  
for a morphism $f:M \to N$ in $\mod_{\ZZ^n} S$, we can define 
$q^*(f): q^*(M) \to q^*(N)$ so that $q^*(f)_\ba: q^*(M)_\ba \to q^*(N)_\ba$ 
coincides with $f_{q(\ba)}: M_{q(\ba)} \to N_{q(\ba)}$.  Clearly, this 
construction gives the functor $q^* : \mod_{\ZZ^n} S \to \mod_{\ZZ^n} S$. 
By \cite[Lemma~2.4]{BF}, this is an exact functor. 

Let $i \in [n]$ and $j \in \ZZ$. 
Define the order preserving map $\sigma_{\<i,j\>} :\ZZ^n \to \ZZ^n$ 
by $(\sigma_{\<i,j\>}(\ba))_k=a_k$ for all $k \ne i$ and 
$$
(\sigma_{\<i,j\>}(\ba))_i=\begin{cases}
a_i -1 & \text{if $a_i \geq j$,}\\
a_i & \text{if $a_i < j$.}
\end{cases}
$$
We denote the functor $(\sigma_{\<i,j\>})^* : \mod_{\ZZ^n} S \to \mod_{\ZZ^n} S$ 
by $(-)^{\tl \< i,j\>}$. 

We also define the map $\tau_{\<i,j\>}: \ZZ^n \to \ZZ^n$ by 
$\tau_{\<i,j\>}(\ba)_k = a_k$ for all $k \ne i$ and 
$$\tau_{\<i,j\>}(\ba)_i = \begin{cases}
a_i+1 & \text{if $a_i \geq j$,}\\
a_i & \text{if $a_i < j$.}
\end{cases}$$
If there is no danger of confusion, we simply denote $\sigma_{\<i, j\>}$ and $\tau_{\<i,j\>}$ 
by $\sigma$ and $\tau$ respectively. 
Note that $\sigma \circ \tau = \opn{Id}$, but $\tau \circ \sigma \ne \opn{Id}$ 
(in fact, if $a_i=j-1$, then $\sigma(\ba)=\sigma(\ba+\be_i)=\ba$). 
We have  $(M^{\tl \<i,j\>})_\ba = M_{\sigma(\ba)}$  and $M_\ba = (M^{\tl \<i,j\>})_{\tau(\ba)}$ 
for all $\ba \in \ZZ^n$.  However, if $a_i=j-1$ (in this case, $\sigma(\ba)=\tau(\ba)=\ba$), 
then  $M_\ba = (M^{\tl \<i,j\>})_{\ba} = (M^{\tl \<i,j\>})_{\ba+\be_i}$  
and the multiplication map 
$(M^{\tl \<i,j\>})_\ba \ni y \longmapsto x_i y \in (M^{\tl \<i,j\>})_{\ba+\be_i}$ is bijective. 
It is easy to see that $S(-\ba)^{\tl \<i,j\>} 
\cong S(-\tau(\ba))$. 

Let $I = (x^{\ba_1}, \ldots, x^{\ba_r})$ be a monomial ideal. 
If $j<0$, then $(S/I)^{\tl \<i,j\>} \cong (S/I)(-\be_i)$. 
So, when we treat $(S/I)^{\tl \<i,j\>}$ or $I^{\tl \<i,j\>} $,  
we assume that $j>0$.  Since $S^{\tl \<i,j\>}\cong S$ in this case  
and $(-)^{\tl \<i,j\>}$ is an exact functor,  we have 
$(S/I)^{\tl \<i,j\>} \cong S/I^{\tl \<i,j\>}$ and 
$I^{\tl \<i,j\>} = (x^{\tau(\ba_1)}, \ldots, x^{\tau(\ba_r)})$. 

\begin{rem}
In \cite{OY}, Okazaki and the author constructed the functor 
$(-)^{\tl \bb} : \mod_{\NN^n} S \to \mod_{\NN^n} S$ for $\bb \in \NN^n$.  
(In the context there, the restriction to $\mod_{\NN^n} S$ is essential. 
However, we can ignore this now.) It is easy to see that 
$(-)^{\tl \be_i} \cong (-)^{\tl \<i, 1\>}$ and 
$(-)^{\tl \bb}= ((-)^{\tl \<1,1\>})^{b_1} \circ \cdots \circ
((-)^{\tl \<n,1 \>})^{b_n}.$
In this sense, $(-)^{\tl \<i, j\>}$ is a generalization of $(-)^{\tl \bb}$.  
\end{rem}

\begin{prop}[{c.f. \cite[Proposition~5.2]{OY}}]\label{triangles preserve}
For $M \in \mod_{\ZZ^n} S$, we have 
$$\beta_{l, \, \ba} (M) = \beta_{l, \, \tau(\ba)} 
(M^{\tl \<i, j\>}) \ \text{for all $l \in \NN$ and $\ba \in \ZZ^n$,}$$
$$\depth M^{\tl \<i,j\>}= \depth M \quad \text{and} \quad 
\Ass M^{\tl \<i,j \>}= \Ass M.$$
In particular, $\dim_S M^{\tl \<i, j\>}= \dim_S M.$
\end{prop}

\begin{proof}
Let  $F_\bullet$ be a $\ZZ^n$-graded minimal free resolution of $M$. 
Since  $(-)^{\tl \<i, j\>}$ is exact and $S(-\ba)^{\tl \<i,j\>} \cong S(-\tau(\ba))$, 
$(F_\bullet)^{\tl \<i, j\>}$ is a minimal free resolution of 
$M^{\tl \<i, j\>}$.  Hence the equations on the Betti numbers and  depths hold.

Recall that an associated prime of a $\ZZ^n$-graded module is always a monomial ideal. 
Assume that $\fp \in \Ass M^{\tl \<i, j\>}$. 
Then we can take a homogeneous element $y \in (M^{\tl \<i, j\>})_\ba$ 
with $\ann(y)=\fp$. 
It is easy to see that the element 
$y' \in M_{\sigma(\ba)}$ corresponding to $y$ 
satisfies  $\ann(y')=\fp$. Hence  $\fp \in \Ass M$. 
Conversely, assume that $\fp \in \Ass M$ and $\ann(y)=\fp$ for $y \in M_\ba$. 
If either $x_i \not \in \fp$ or $a_i \ne j-1$, the element $\dot{y} \in 
(M^{\tl \<i, j\>})_{\tau(\ba)}$ corresponding to $y$ satisfies 
$\ann(\dot{y})=\fp$. So let us consider the case  $x_i \in \fp$ and $a_i=j-1$. 
Then $\tau(\ba)=\ba$ and $(M^{\tl \<i,j\>})_{\ba} = (M^{\tl \<i,j\>})_{\ba+\be_i}
=M_\ba$, but the element $\dot{y} \in 
(M^{\tl \<i, j\>})_{\ba+\be_i}$ corresponding to 
$y$ satisfies $\ann(\dot{y})=\fp$. Hence $\fp \in \Ass M^{\tl \<i, j\>}$.
\end{proof}

The following is an immediate consequence of the above proposition. 

\begin{cor}\label{CM & Gor}
For $M \in \mod_{\ZZ^n} S$, $M$ is Cohen-Macaulay, if and only if so is 
$M^{\tl \<i, j\>}$.  For a monomial ideal $I$, $S/I$ is Gorenstein 
if and only if so is $S/I^{\tl \<i,j\>}$.  
\end{cor}

Let $\lambda_{\<i,j\>}: \ZZ^n \to \ZZ^n$ be the order preserving map defined by
$(\lambda_{\<i,j\>}(\ba))_k =a_k$ for all $k \ne i$ and  
$$
(\lambda_{\<i,j\>}(\ba))_i:=
\begin{cases}
a_i  & \text{if $a_i  \geq j$,} \\
a_i+1 & \text{if $a_i < j$.}
\end{cases}
$$ 
Let $(-)^{\tr \<i, j\>}$ denote the functor $\lambda^*:\mod_{\ZZ^n} S \to \mod_{\ZZ^n} S$.

We define $\rho_{\<i,j\>}:\ZZ^n \to \ZZ^n$ by $(\rho_{\<i,j\>}(\ba))_k=a_k$ for 
all $k \ne i$ and 
$$
(\rho_{\<i,j\>}(\ba))_i:=
\begin{cases}
a_i & \text{if $a_i > j$,}\\
a_i - 1 & \text{if $a_i \leq j$,} 
\end{cases}
$$ 
Then we have $S(-\ba)^{\tr \<i, j\>} \cong 
S(-\rho_{\<i,j\>}(\ba))$ for all $\ba \in \ZZ^n$. 

\begin{rem}\label{toukasei}
It is easy to see that  $M^{\tl \<i,j\>}\cong (M^{\tr \<i, j-1 \>})(-\be_i)$. 
In this sense, $(-)^{\tr \<i,j\>}$ is  ``parallel" to $(-)^{\tl \<i,j\>}$.
Hence Proposition~\ref{triangles preserve} and the first assertion of Corollary~\ref{CM & Gor} 
also hold for $M^{\tr \<i,j\>}$. 
\end{rem}

Let $\Db(\mod_{\ZZ^n} S)$ denote the bounded derived category. 
Since the functor $(-)^{\tl  \<i,j\>}: \mod_{\ZZ^n} S \to \mod_{\ZZ^n} S$ is exact, 
it can be extended to $(-)^{\tl  \<i,j\>}: \Db(\mod_{\ZZ^n} S) \to 
\Db(\mod_{\ZZ^n} S)$ in the natural way. The same is true for  $(-)^{\tr  \<i,j\>}$. 
We denote the exact functor 
${\mathbf R}\!\Hom_S(-,S):\Db(\mod_{\ZZ^n} S) \to \Db(\mod_{\ZZ^n} S)^\op$ by $\DD$.

\begin{thm}\label{local duality}
We have natural isomorphisms 
$$\DD \circ (-)^{\tl  \<i,j\>} \cong 
(-)^{\tr  \<i, -j\>} \circ \DD \quad \text{and} \quad  
\DD \circ (-)^{\tr  \<i,j\>} \cong (-)^{\tl  \<i, -j\>} \circ \DD.$$
In particular, we have $$\Ext_S^l(M^{\tl  \<i,j\>}, S) \cong 
\Ext_S^l(M, S)^{\tr  \<i,-j\>} \quad \text{and} \quad 
\Ext_S^l(M^{\tr  \<i,j\>}, S) \cong 
\Ext_S^l(M, S)^{\tl  \<i,-j\>}$$ 
for $M \in \mod_{\ZZ^n} S$. 
\end{thm}

\begin{proof}
For  $M^\bullet \in \Db(\mod_{\ZZ^n} S)$, we can take its minimal free resolution $F^\bullet$.   
Then $C_1^\bullet := \Hom_S((F^\bullet)^{\tl \<i,j\>},S)$ 
(resp. $C_2^\bullet := \Hom_S(F^\bullet,S)^{\tr \<i, -j\>}$) is a minimal free resolution 
of $\DD((M^\bullet)^{\tl \<i,j\>})$ 
(resp. $\DD(M^\bullet)^{\tr \<i,-j\>}$). 
If the $l^{\rm th}$ term $F^l$ of $F^\bullet$ is of the form 
$\bigoplus_{\ba \in \ZZ^n} S(-\ba)^{b_{l,\ba}}$,  
then both $C_1^l$ and $C_2^l$ are isomorphic to 
$\bigoplus_{\ba \in \ZZ^n} S(\tau_{\<i,j\>}(\ba))^{b_{-l,\ba}}$. 
(Note that $-\tau_{\<i,j\>}(\ba) = \rho_{\<i,-j\>}(-\ba)$ for all $\ba \in \ZZ^n$.)   
Considering the relation among 
the matrices representing the differential maps $C_1^l \to C_1^{l+1}$,  
$C_2^l \to C_2^{l+1}$ and $F^{-l-1} \to F^{-l}$, we can easily show that 
$C_1^\bullet$ and $C_2^\bullet$ are isomorphic. 
Hence we have  $\DD((M^\bullet)^{\tl  \<i,j\>}) 
\cong \DD(M^\bullet)^{\tr  \<i,-j\>}$. 
The remaining statements are clear now.    
\end{proof}

Let $M$ be a finitely generated $S$-module. 
We say $M$ is {\it sequentially Cohen-Macaulay} if $\Ext_S^{n-i}(M,S)$ 
is either a Cohen-Macaulay module of dimension $i$ or the 0 module for all $i$. 
The original definition is given by the existence of a certain filtration 
(see \cite[III, Definition~2.9]{St}), 
however it is equivalent to the above one by \cite[III, Theorem~2.11]{St}.

Recall that we say  a quotient ring $R:=S/I$ satisfies Serre's condition 
$(S_r)$ if $\opn{depth} R_{\fp} \geq \min \{ \, r, \, \dim R_{\fp} \, \}$ 
for all prime ideal $\fp$ of $R$. If $R$ satisfies $(S_2)$, then all 
associated primes of $I$ have the same codimension in $S$. Hence, for $r \geq 2$, 
$R$ satisfies $(S_r)$ if and only if 
$\dim_S \Ext^{n-i}_S(R,S) \leq i-r$ for all $i < \dim R$.  
Here the dimension of the 0 module is $-\infty$. 

\begin{cor}\label{seqCM for slide}
For $M \in \mod_{\ZZ^n} S$, it is sequentially Cohen-Macaulay, if and only if 
so is $M^{\tl  \<i,j\>}$, if and only if 
so is $M^{\tr  \<i,-j\>}$.   
For a monomial ideal $I$, $S/I$ satisfies $(S_r)$ if and only if so dose $S/I^{\tl  \<i,j\>}$. 
\end{cor}

\begin{proof}
By Proposition~\ref{triangles preserve}, Remark~\ref{toukasei} and 
Theorem~\ref{local duality},
the assertion follows from the remarks before this corollary.  
(Precisely speaking, the $(S_1)$ condition is not covered by this argument. 
However this case immediately follows from Proposition~\ref{triangles preserve}.)
\end{proof}

\section{Sliding functors and Stanley depth}
For a subset $Z \subset \{ x_1, \ldots, x_n \}$, $\kk[Z]$ denotes 
the subalgebra of $S$ generated by all $x_i \in Z$. 
Clearly, $\kk[Z]$ is a polynomial ring with $\dim \kk[Z] = \# Z$. 
Let $M \in \mod_{\ZZ^n} S$.  We say the $\kk[Z]$-submodule $m \, \kk[Z]$ of $M$ 
generated by a homogeneous element $m \in M_\ba$ is a {\it Stanley space}, if it is 
$\kk[Z]$-free.  
A {\it Stanley decomposition} $\cD$ of $M$ is a presentation of 
$M$ as a finite direct sum of Stanley spaces. That is, 
$$\cD:\bigoplus_{l=1}^s m_l \, \kk[Z_l] = M$$ 
as $\ZZ^n$-graded $\kk$-vector spaces, where each $m_l \, \kk[Z_l]$ 
is a Stanley space. 

Let $\sd(M)$ be the set of Stanley decompositions of $M$. 
For all $0 \ne M \in \mod_{\ZZ^n} S$, we have $\sd (M) \ne \emptyset$. 
For $\cD=\bigoplus_{l=1}^s m_l \, \kk[Z_l] \in \sd (M)$, we set 
$$\sdepth (\cD) := \min\set{\# Z_l}{l = 1,\dots ,s},$$
and call it the  {\em Stanley depth} of $\cD$. 
The Stanley depth of $M$ is defined by 
$$
\sdepth (M) := \max\set{\sdepth \cD}{\cD \in \sd (M)}.
$$

This invariant behaves somewhat strangely. For example, contrary to the usual depth, 
no relation between the Stanley depth of a monomial ideal $I$ and that of the 
quotient $S/I$ is known.

The following conjecture, which is widely open even for monomial ideals $I$ and the quotients $S/I$, 
is a leading problem of this subject. See \cite{A1,A2,HVZ} for further information (especially, known results mentioned below). 

\begin{conj}
[Stanley]\label{Stanley conj}
For any $M \in \mod_{\ZZ^n} S$, we have
$$
\sdepth M \ge \depth M.
$$
\end{conj}

The following is the main result of this section. 

\begin{thm}\label{sdepth}
For $M \in \mod_{\ZZ^n} S$, we have 
$$\sdepth M = \sdepth M^{\tl \<i,j\>}.$$
Hence Stanley's conjecture holds for $M$ if and only if 
it holds for $M^{\tl \<i,j\>}$. 
\end{thm}

\begin{proof}
Let $M^{\tl \<i,j\>}= \bigoplus_{l=1}^s m_l \, \kk[Z_l]$ be a Stanley decomposition with 
$m_l \in (M^{\tl \<i,j\>})_{\ba_l}$. 
We may assume that; 
\begin{itemize}
\item If $1 \leq l \leq q$, then $(\ba_l)_i=j-1$ and $x_i \in Z_l$;
\item If $q+1 \leq l \leq r$, then $(\ba_l)_i \ne j-1$;
\item If $r+1 \leq l \leq s$, then $(\ba_l)_i = j-1$ and $x_i \not \in Z_l$. 
\end{itemize}
Recall that if $a_i= j-1$ then 
$(M^{\tl\<i,j\>})_\ba = (M^{\tl\<i,j\>})_{\ba+\be_i}=M_\ba$ and  
the multiplication map $(M^{\tl \<i,j\>})_\ba \ni 
y \longmapsto x_i y \in (M^{\tl \<i,j\>})_{\ba+\be_i}$ is bijective. 
For $\bb \in \NN^n$ with $b_i \ne j-1$, we have 
$$(M^{\tl \<i,j\>})_\bb \subset \bigoplus_{l=1}^q (x_i m_l) \kk[Z_l] \oplus 
\bigoplus_{l=q+1}^r  m_l \kk[Z_l].$$ 
For each $1 \leq l \leq r$, let $m_l' \in M_{\sigma(\ba_l)}$ be the element corresponding to $m_l$. 
Then we can check that $M= \bigoplus_{l=1}^{r} m_l' \, \kk[Z_l]$ is a Stanley decomposition. 
Hence we have $\sdepth M \geq \sdepth M^{\tl \<i,j\>}.$

Conversely, let $M= \bigoplus_{l=1}^s m_l \, \kk[Z_l]$ be a Stanley decomposition with 
$m_l \in M_{\ba_l}$. We may assume that 
\begin{itemize}
\item If $ 1 \leq l \leq t$, then either  $(\ba_l)_i \ne  j-1$ or $x_i \in Z_l$. 
\item If $t+1 \leq l \leq s$, then   $(\ba_l)_i =  j-1$ and $x_i \not \in Z_l$. 
\end{itemize}
For $1 \leq l \leq s$, let $\dot{m}_l \in (M^{\tl \<i,j\>})_{\tau(\ba_l)}$ 
be the element corresponding to $m_l$. 
Further more, for $1 \leq l \leq t$,  
set $\dot{m}_{s+l} := x_i \, \dot{m}_l \in  (M^{\tl \<i,j\>})_{\tau(\ba_l)+\be_i}$  
(note that $\tau(\ba_l)=\ba_l$ in this case) and $Z_{l+s} :=Z_l$. 
Now we have homogeneous elements $\dot{m}_1, \ldots, \dot{m}_{s+t}$ 
of $M^{\tl \<i,j\>}$, and it is easy to check that 
$M^{\tl \<i,j\>}= \bigoplus_{l=1}^{s+t} \dot{m}_l \, \kk[Z_l]$. 
Hence we have $\sdepth M \leq \sdepth M^{\tl \<i,j\>}.$
\end{proof}

Unfortunately (?), many classes of monomial ideals for which Stanley's conjecture  has been proved 
are closed under the operation $(-)^{\tl \<i,j\>}$. 
For example, a monomial ideal $I$ is Borel type 
(resp. generic, cogeneric) if and only if so is $I^{\tl \<i,j\>}$. 
The following class is an exception.  

Let $I$ be a monomial ideal minimally generated by monomials $m_1, \ldots, m_r$.  
We say $I$ has {\it linear quotients} if after suitable change  of the order of $m_i$'s,   
the colon ideal $(m_1, \ldots, m_{i-1}):m_i$ is a monomial prime ideal for all $2 \leq i \leq r$. 
It is known that Stanley's conjecture holds for ideals of this type (\cite{HVZ}).   
Note that the ideal  $I:=(x^r, x^{r-1}y, \cdots, xy^{r-1}, y^r) \subset \kk[x,y]$ has linear quotients.  
Let $c_1, \ldots, c_r$ and  $d_1, \ldots, d_r$ be strictly 
decreasing sequences consisting of positive integers.    
Applying the functors of the form $(-)^{\tl \<i,j\>}$ to $I$ repeatedly, we get 
$I'=(x^{c_1}, x^{c_2}y^{d_r}, \ldots, x^{c_r}y^{d_2}, y^{d_1})$, 
which does not have linear quotients.  
Of course, the conjecture is trivial for $I'$, but it is easy to construct 
non-trivial examples. 

\medskip

Let $I$ be a monomial ideal minimally generated by $x^{\ba_1}, \ldots, x^{\ba_r}$. 
Consider the following condition:   
\begin{itemize}
\item[$(*)$] For each $i \in [n]$, there is a positive integer $b_i$ such that
$\{\, (\ba_l)_i \mid 1 \le l \le r \, \}$ is  $\{ \, 1, 2, \ldots, b_i \, \}$  
or $\{ \, 0,1, \ldots, b_i \, \}$.   
\end{itemize}

\begin{cor}\label{consecutive}
If Stanley's conjecture holds for all monomial ideals satisfying the condition $(*)$, 
then the conjecture holds for arbitrary monomial ideals.   
The same is true for the quotient rings by monomial ideals.  
\end{cor}

\begin{proof}
Let $J$ be a monomial ideal minimally generated by $x^{\bb_1}, \ldots, x^{\bb_r}$. 
If the variable $x_1$ divides none of $x^{\bb_1}, \ldots, x^{\bb_r}$, 
then we can reduce the problem to the ideal $J \cap \kk[x_2, \ldots, x_n]$ of $\kk[x_2, \ldots, x_n]$. 
So we may assume that each $x_i$ divides at least one of $x^{\bb_1}, \ldots, x^{\bb_r}$. 
Then there is a monomial ideal $I$ satisfying $(*)$ from which 
$J$ can be produced by the iterated applications of the functor $(-)^{\tl \<i,j\>}$.  
\end{proof}

\section{Polarization functor}
In the rest of the paper, we fix $\ba \in \NN^n$ such that $a_i \geq 1$ for all $i$. 

\begin{dfn}[{Miller \cite{M}}]
Let $\ba \in \ZZ^n$. 
We say $M \in \mod_{\NN^n} S$ is {\it positively $\ba$-determined}, if the multiplication map 
$M_\bb \ni y \longmapsto x_i y \in M_{\bb+\be_i}$ is bijective for all 
$\bb \in \NN^n$ and $i \in [n]$ with $b_i \geq a_i$. Let $\mod_{\ba} S$ denote 
the full subcategory of $\mod_{\NN^n} S$ consisting of positively $\ba$-determined modules. 

Set  $\one :=(1, \ldots, 1) \in \NN^n$. A positively $\one$-determined module 
is called a {\it squarefree module}, and we denote $\mod_\one S$ by $\Sq S$.  
\end{dfn}

The category $\mod_\ba S$ is an abelian category with 
enough projectives. An indecomposable projective 
is $S(-\bb)$ for some $\bb \in \NN^n$ with $\bb \preceq \ba$. 
If $M \in \mod_\ba S$, then $M^{\tl \<i,j\>} \in \mod_{\ba+\be_i} S$. 

For $\ba \in \NN^n$, set $|\ba|:=\sum_{i=1}^n a_i$ and let $\wS$ be the polynomial ring
$$\kk[\, x_{i,j} \mid 1 \leq i \leq n, 1 \leq j \leq a_i \,]$$
of $|\ba|$ variables. 
Extending the classical technique\ of the polarization of monomial ideals, 
we will define the functor $\pol: \mod_{\ba} S \to \Sq \wS$. 
This idea has already appeared in \cite[Theorem~2.1]{BH} and 
\cite[\S4]{Sb} in slightly different setting. 
See Remark~\ref{Sbarra} below for precise information.  

Set $[\ba] := \{ \, (i,\, j) \mid 1 \leq i \leq n, 1 \leq j \leq a_i \, \}$. 
Then $\NN^{[\ba]}$ is the set of vectors  
$\bb=( \, b_{i,j} \mid  1 \leq i \leq n, 1 \leq j \leq a_i \, )$
with $b_{i,j} \in \NN$ for all $i,j$.   
We can regard $\wS$ as an $\NN^{[\ba]}$-graded ring in the natural way.  
Define the order preserving map $\eta: \NN^{[\ba]} \to \NN^n$  by 
$$
\eta(\bb)_i = \begin{cases}
a_i & \text{if $b_{i,j}>0$ for all $j$,}\\
\min \{ \, j-1 \mid b_{i,j} =0 \, \} & \text{otherwise.} 
\end{cases}
$$
For $M \in \mod_{\ba} S$, we can construct a new module 
$\eta^*(M) \in \mod_{\NN^{[\ba]}} \wS$ so that $\eta^* (M)_\bb \cong M_{\eta(\bb)}$ 
and the multiplication map $\eta^*(M)_{\bb} \ni y \longmapsto t^{\bc} y \in \eta^*(M)_{\bb+\bc}$ 
is given by $M_{\eta(\bb)} \ni y \longmapsto x^{\eta(\bb+\bc)-\eta(\bb)} y \in M_{\eta(\bb+\bc)}$ 
for all $\bb,\bc \in \NN^{[\ba]}$. 
It is easy to show that $\eta^*(M) \in \Sq \wS$.  Hence $\eta^*$ gives a 
functor $\mod_{\ba} S \to \Sq \wS$, which is exact by 
an argument similar to  \cite[Lemma~2.4]{BF}. 
 
For $\bb \in \NN^n$ with $\bb \preceq \ba$, we define $\widetilde{\bb} \in \NN^{[\ba]}$ by 
$$\widetilde{b}_{i,j}=\begin{cases}
1 & \text{if $j \leq b_i$,}\\
0 & \text{otherwise.}
\end{cases}$$
Then it is easy to check that $\eta^*(S(-\bb)) \cong \wS(-\widetilde{\bb})$.  
For a positively $\ba$-determined monomial ideal $I=(x^{\bb_1}, \ldots, x^{\bb_r}) \subset S$, 
$\eta^*(I)$ coincides with the polarization 
$\pol(I)=(x^{\widetilde{\bb}_1}, \ldots, x^{\widetilde{\bb}_r}) \subset \wS$ of $I$. 
So we denote the functor $\eta^*$ by $\pol$. 

\begin{lem}\label{polarization basic}
For $M \in \mod_\ba S$, we have 
$$\beta_{i,j}^{\wS}(\pol(M)) = \beta_{i,j}^S(M),$$
$$\opn{depth}_{\wS} \, (\pol(M)) = \depth M+ |\ba|-n,$$ 
$$\dim_{\wS} \, (\pol(M) )= \dim_S M + |\ba|-n,$$       
and
$$\deg_{\wS} (\pol(M))=\deg_S M.$$ 
\end{lem}

\begin{proof}
Let $F_\bullet$ be a $\ZZ^n$-graded minimal free resolution of $M$. 
Note that any indecomposable summand of $F_i$ is of the form 
$S(-\bb)$ for some $\bb \preceq \ba$.  
By the exactness of the functor $\pol$ and that $\pol(S(-\bb)) \cong \wS(-\widetilde{\bb})$, 
$\pol(F_\bullet)$ is a $\ZZ^{[\ba]}$-graded minimal free resolution of $\pol(M)$.  
Hence the assertions on Betti numbers and depths are clear. 
Moreover, we can check that $\pol(M)$ is the same thing as the ``{\it lifted module}" constructed in 
\cite[Theorem~2.1]{BH}.  Clearly, 
\begin{equation}\label{Theta}
\Theta := \{ x_{i,1}-x_{i,j} \mid 1 \leq i \leq n, \, 2 \leq j \leq a_i \, \}
\end{equation}
forms a regular sequence in $\wS$ of length $|\ba|-n$, and $\wS/(\Theta) \cong S$ 
in the natural way. 
By \cite[Theorem~2.1]{BH}, $\Theta$ also forms a $\pol(M)$-regular sequence, and  
$\pol(M) \otimes_{\wS} \wS/(\Theta) \cong M$.  
The assertions on dimensions and degrees follow from this fact. 
\end{proof}

Next we construct the ``reversed" version of $\pol$. 
Define the map $\iota: \NN^{[\ba]} \to \NN^n$  by 
$$
\iota(\bb)_i = \begin{cases}
a_i & \text{if $b_{i,j}>0$ for all $j$,}\\
a_i -\max \{ \, j \mid b_{i,j} =0 \, \} & \text{otherwise.}
\end{cases}
$$
As $\eta:\NN^{[\ba]} \to \NN^n$ gave the functor $\pol$, the map 
$\iota: \NN^{[\ba]} \to \NN^n$ defines the functor $\dpol:\mod_\ba S \to \Sq \wS$. 
For $\bb \in \NN^n$ with $\bb \preceq \ba$, we take 
$\widehat{\bb} \in \wS$ so that 
$$\widehat{b}_{i,j}=\begin{cases}
1 & \text{if $j \geq a_i- b_i+1$,}\\
0 & \text{otherwise.}
\end{cases}$$
Then we have $\dpol(S(-\bb))\cong \wS(-\widehat{\bb})$. 
Clearly, $\dpol$ is the same thing as $\pol$ modulo suitable exchange of variables 
of $\wS$. Hence Lemma~\ref{polarization basic} also holds for $\dpol$.  
Moreover, $\Theta \subset \wS$ of \eqref{Theta} also forms a $\dpol(M)$-regular sequence, 
and $\dpol(M) \otimes_{\wS} \wS/(\Theta) \cong M$.  

\begin{exmp}\label{confusing}
Polarization does not commute with degree shifts. 
For example, set $S:=\kk[x,y]$, $\wS:=\kk[x_1,x_2,y_1,y_2,y_3]$ and $\ba := (2,3)$. 
Then $(S/(x^2,y))(-(0,1)) \in \mod_\ba S$ and 
$$\opn{\mathsf pol}_\ba((S/(x^2,y))(-(0,1)))=\wS/(x_1x_2, y_2)(-\be),$$
where $\be \in \NN^{[\ba]}$ is the unit vector corresponding to $y_1$. 
Clearly, this is not isomorphic to neither $\opn{\mathsf pol}_\ba((S/(x^2,y))$ 
nor $\opn{\mathsf pol}^\ba((S/(x^2,y))$, even if we forget the grading. 
\end{exmp}

Since $\Ext_S^i(M,S(-\ba)) \in \mod_\ba S$ for all $M \in \mod_\ba S$, 
${\mathbf R}\Hom_S(-,S(-\ba))$ gives a functor 
$\Db(\mod_\ba S) \to \Db(\mod_\ba S)^\op$, which is denoted by $\DD_S$. 
Similarly, $\DD_{\wS}$ 	denotes the duality functor 
${\mathbf R}\Hom_{\wS}(-,\wS(-\one)): 
\Db(\Sq \wS) \to \Db(\Sq \wS)^\op$.  Here $\one \in \NN^{[\ba]}$ is the 
vector whose coordinate are all 1. Note that the exact functors $\pol$ and $\dpol$ 
can be extended to the functors $\Db(\mod_\ba S) \to \Db(\Sq \wS)$. 

\begin{thm}[{c.f. Sbarra \cite[Corollary~4.10]{Sb}}]
\label{local duality for polarization}
We have a natural isomorphism 
$$\DD_{\wS} \circ \pol \cong \dpol \circ \DD_S.$$ 
In particular, 
$$\Ext_{\wS}^i(\pol(M), \wS(-\one)) \cong \dpol(\Ext_S^i(M,S(-\ba)))$$
for all $M \in \mod_\ba S$. 
\end{thm}

\begin{proof}
We will prove the assertion by a similar way to Theorem~\ref{local duality}. 
It suffices to show 
$\DD_{\wS} \circ \pol (F^\bullet) \cong \dpol \circ \DD_S (F^\bullet)$ 
for $F^\bullet \in \Db(\mod_\ba S)$ such that each $F^i$ is a free $S$-module. 
If $S(-\bb) \in \mod_\ba S$, then we have 
$\DD_{\wS}(\pol(S(-\bb)))= \DD_{\wS}(\wS(-\widetilde{\bb}))=\wS(-(\one-\widetilde{\bb}))$
and $\dpol(\DD_S(S(-\bb))) = \dpol(S(-(\ba-\bb)))= \wS(-\widehat{(\ba-\bb)})$.  
However, $\one-\widetilde{\bb} = \widehat{(\ba-\bb)}$ holds. In fact, for all $i,j$, we have  
$$(\one-\widetilde{\bb})_{i,j} = (\widehat{\ba-\bb})_{i,j}
=\begin{cases}
1 & \text{if $b_i < j \leq a_i$,}\\
0 & \text{if $1 \leq j \leq b_i$.}
\end{cases}
$$ 
Hence we have $\DD_{\wS}(\pol(F^\bullet))^i \cong \dpol (\DD_S(F^\bullet))^i$ for all $i$. 
It is easy to see that these isomorphisms induce an isomorphism of the complexes.    
\end{proof}

\begin{cor}\label{Cor for polarization}
For $M \in \mod_\ba S$, 
it is sequentially Cohen-Macaulay if and only if 
so is $\pol(M)$. The same is true for  Serre's condition $(S_r)$ of a monomial quotient $S/I$.      
\end{cor}

\begin{proof}
By Lemma~\ref{polarization basic} and Theorem~\ref{local duality for polarization}, 
the same argument as Corollary~\ref{seqCM for slide} works. 
\end{proof}

The {\it arithmetic degree} $\adeg_S(M)$ of a finitely generated $S$-module $M$ 
is the ``degree" reflecting the contribution of all associated primes of $M$ 
(the usual degree $\deg(M)$ only concerns minimal primes $\fp$ of $M$ with 
$\dim S/\fp = \dim M$). In \cite[Proposition~1.11]{V}, it is shown that 
$$\adeg_S M = \sum_{i=0}^n d_i(M),$$
where
$$d_i(M)= \begin{cases}
\deg_S \Ext_S^{n-i}(M,S) & \text{if $\dim_S \Ext_S^{n-i}(M,S)=i$,}\\
0 & \text{otherwise.}
\end{cases}$$

\begin{cor}[{c.f. Fr\"{u}bis-Kruger and Terai~\cite{FT}}]\label{adeg}
For $M \in \mod_{\ba} S$, we have 
$$\adeg_S (M) = \adeg_{\wS} (\pol(M)).$$
\end{cor}

\begin{proof}
The assertion follows from 
Lemma~\ref{polarization basic} and Theorem~\ref{local duality for polarization}.  
\end{proof}

When $M=S/I$ for a monomial ideal $I$, Corollary~\ref{adeg} has been given by 
Fr\"{u}bis-Kruger and Terai \cite{FT}.  
Our proof also works for the {\it homological degree} $\hdeg_S (M)$ of $M$, 
which is defined by the iterated use of the operation $\Ext_S^i(-,S)$ 
(see \cite{V} for detail).  Hence the following equation holds;
$$\hdeg_S (M) = \hdeg_{\wS} (\pol(M)).$$  

\begin{rem}\label{Sbarra}
(1) As mentioned before, 
Bruns and Herzog \cite{BH} gave the same construction as  $\pol :\mod_\ba S \to \Sq \wS$, 
but they did not recognize it as a functor. 
On the other hand, Sbarra \cite{Sb} treated a polarization as a functor. 
However he used the convention something like 
$\opn{\mathsf{pol}}_{\pm \be_i}:\mod_{\ZZ^n} S \to \mod_{\ZZ^n} S'$ with $S'=S[x_i']$.  
Roughly speaking, his functor is similar to our $\pol$ (resp. $\dpol$) 
in the positive (resp. negative) degree parts. Hence he did not need the pair $\pol$ 
and $\dpol$, while \cite[Corollary~4.10]{Sb} corresponds to our 
Theorem~\ref{local duality for polarization}.   

(2) For a monomial ideal $I \subset S$, 
we say a squarefree monomial ideal $J$ of 
$\wS=\kk[\, x_{i,j} \mid 1 \leq i \leq n, 1 \leq j \leq a_i \,]$ 
is a {\it generalized polarization}, 
if  $\Theta \subset \wS$ of \eqref{Theta}  
forms a $\wS/J$-regular sequence, and $\wS/J \otimes_{\wS} \wS/(\Theta) \cong S/I$ 
through the natural identification $\wS/(\Theta) \cong S$. 
(If the latter condition is satisfied, the former is equivalent to 
the condition that $\beta_{i,j}^S(I)=\beta_{i,j}^{\wS}(J)$ for all $i,j$.)  
Of course, any $I$ has the standard polarization $\pol(I)$ (or $\dpol(I)$), but  
some ideal has alternative one. 

For example, let $I:=(x^2y, x^2z, xyz, xz^2, y^3, y^2z,yz^2)$ be an ideal of $S:=\kk[x,y,z]$.   
Then the ideal 
$$J=(x_1x_2y_3, x_1x_2z_3,x_1y_2z_3, x_1z_2z_3, y_1y_2y_3, y_1y_2z_3,y_1z_2z_3)$$
of $\wS:=\kk[x_1,x_2,y_1, y_2,y_3,z_1,z_2,z_3]$ is a generalized polarization of $I$ (the variable $z_1$ 
does not appear in the generators of $J$, but we prepare it for the clean presentation). 
{\it Macaulay2} computation shows that $\deg \Ext^3_S(S/I,S)=6$ and  
$\deg \Ext^3_{\wS}(\wS/J,\wS)=5$. Hence Theorem~\ref{local duality for polarization}, 
Corollaries~\ref{Cor for polarization} and \ref{adeg} might fail for non-standard polarizations.       

The author was told generalized polarizations of the above type by G. Fl\o ystad.   
\end{rem}

\section{Relation to Bier-Murai sphere}
For $i \in [n]$, set 
$$S':= \kk[x_1, \ldots, x_{i-1}, x_{i+1}, \ldots, x_n, x_{i,1}, x_{i,2}] 
\cong S[x_i'].$$  
Then we can consider the functor 
$\opn{\mathsf{pol}}_{\one+ \be_i} \circ (-)^{\tl \<i, 1 \>} : \Sq S \to 
\Sq S'$. This functor sends a squarefree monomial ideal $I \subset S$ to 
the squarefree  monomial ideal $I' \subset S'$ given by replacing $x_i$ 
(appearing in the minimal generators of $I$) 
by $x_{i,1} \cdot x_{i,2}$.  Let $\Delta$ (resp. $\Delta'$) be a simplicial complex 
whose Stanley-Reisner ideal is $I$ (resp. $I'$). In \cite{BH}, $\Delta'$ is called  
a {\it 1-vertex inflation} of $\Delta$. Hence we denote the functor 
$\opn{\mathsf{pol}}_{\one+ \be_i} \circ (-)^{\tl \<i,1\>}$ by $\infl_i$. 

\begin{lem}
For $M \in \Sq S$, we have 
$\beta_j^S(M)=\beta_j^{S'}(\infl_i(M))$, $\dim_{S'} (\infl_i(M))= \dim_S M+1$, and 
$\opn{depth}_{S'} (\infl_i(M))= \depth M+1$. Moreover, 
$M$ is sequentially Cohen-Macaulay if and only if 
so is $\infl_i(M)$. The same is true for  Serre's condition $(S_r)$ of 
a monomial quotient $S/I$.   
\end{lem}

\begin{proof}
Since the corresponding statements hold for both $(-)^{\tl \<i,1\>}$ and 
$\opn{\mathsf{pol}}_{\one+ \be_i}$ (the latter raises the dimension and depth by 1, 
while the former preserves them), 
the assertion holds for the composition  
$\infl_i = \opn{\mathsf{pol}}_{\one+ \be_i} \circ (-)^{\tl \<1, i \>}$. 
\end{proof}

Miller (\cite{M}) introduced the {\it Alexander duality functor} 
$\sA_\ba : \mod_\ba S \to (\mod_\ba S)^\op$ as follows:   
Let  $M \in \mod_\ba S$, and take $\bb \in \NN^n$ with $\bb \preceq \ba$.  
Set $(\sA_{\ba}(M))_\bb$ to be the dual $\kk$-vector space of $M_{\ba-\bb}$.  
If $b_i \geq 1$, the multiplication map 
$(\sA_\ba(M))_{\bb - \be_i} \ni y \longmapsto x_i y \in (\sA_\ba(M))_\bb$ 
is the $\kk$-dual of $M_{\ba-\bb} \ni z \longmapsto x_i z \in M_{\ba-\bb+\be_i}$. 
 
For a monomial ideal $I \subset S$ with $I \in \mod_\ba S$, 
$I^{\vee\ba} := \sA_\ba(S/I) \in \mod_\ba S$ 
can be regarded as a monomial ideal of $S$. In fact, we have 
$$I^{\vee \ba} 
= ( \, x^\bb \mid \text{$\bb \in \NN^n$, $\bb \preceq \ba$ and $x^{\ba-\bb} \not \in I$} \, ).$$

\begin{lem}\label{Alex dual}
Let $I$ be  a monomial ideal with $I \in \mod_\ba S$. 
For $\bb \in \NN^n$ with $\bb \preceq \ba$, 
$x^\bb$ is a minimal generators of  $I^{\vee \ba}$ if and only if $x^{\tau_{\<i, a_i+2-j\>}(\bb)}$ 
is a minimal generator of  $(I^{\tl\<i,j\>})^{\vee \ba +\be_i}$. 
\end{lem}

\begin{proof}
Clearly, $x^\bb \in I^{\vee \ba}$ if and only if $x^{\ba-\bb} \not \in I$ if and only if 
$x^{\tau_{\<i,j\>}(\ba-\bb)} \not \in I^{\tl \<i, j\>}$ if and only if 
$x^{\ba+\be_i -\tau_{\<i,j\>}(\ba-\bb)} \in 
(I^{\tl \<i,j\>})^{\vee \ba+\be_i}$.  
Since $(\ba+\be_i -\tau_{\<i,j\>}(\ba-\bb))_k = b_k$ for all $k \ne i$ and 
$$
(\ba+\be_i -\tau_{\<i,j\>}(\ba-\bb))_i =\begin{cases}
b_i +1 & \text{if $a_i-b_i  < j$ \, (equivalently, $b_i > a_i -j$),}\\
b_i  & \text{if $a_i-b_i  \geq j$ \, (equivalently, $b_i \leq a_i -j$),}
\end{cases}
$$
we have $\ba+\be_i -\tau_{\<i,j\>}(\ba-\bb)= \tau_{\<i, a_i+1-j \>}(\bb)$. 
Hence $x^\bb \in I^{\vee \ba}$ if and only if 
$x^{\tau_{\<i,a_i+1-j\>}(\bb)} \in (I^{\tl \<i,j\>})^{\vee \ba+\be_i}$.  
However, essentially because $\tau_{\<\bullet, \bullet\>}$ is not surjective, 
$\tau_{\<i, a_i+2-j\>}$ appears. 

Throughout this paragraph, we assume that $b_i= a_i+1-j$. 
Since the multiplication map 
$[I^{\tl\<i,j\>}]_\bc \ni y \longmapsto x_i y \in [I^{\tl\<i,j\>}]_{\bc+\be_i}$  
is bijective for all $\bc \in \NN^n$ with $c_i=j-1$, the multiplication map 
$[(I^{\tl\<i,j\>})^{\vee \ba+\be_i}]_\bb \ni y \longmapsto x_i y \in 
[(I^{\tl\<i,j\>})^{\vee \ba+\be_i}]_{\bb+\be_i}$ is bijective. 
Hence, if $x^{\bb+\be_i} \in (I^{\tl\<i,j\>})^{\vee \ba+\be_i}$, then $x^\bb$ also belong to 
$(I^{\tl\<i,j\>})^{\vee \ba+\be_i}$. It follows that $x^\bb$ is a minimal generator of 
$I^{\vee \ba}$ if and only if it is also a minimal generator of $(I^{\tl\<i,j\>})^{\vee \ba+\be_i}$. 
Since $\tau_{\<i, a_i+2-j\>}(\bb)=\bb$ in this case, the assertion holds. 

In the rest of this proof, we assume that $b_i\ne a_i+1-j$. 
It is easy to see that $x^\bb$ is a minimal generator of $I^{\vee \ba}$ if and only if 
$x^{\tau_{\<i, a_i+1-j\>}(\bb)}$ is a minimal generator of $(I^{\tl\<i,j\>})^{\vee \ba+\be_i}$. 
Since we have  $\tau_{\<i, a_i+1-j\>}(\bb)=\tau_{\<i, a_i+2-j\>}(\bb)$ now, we are done.
\end{proof}

In the rest of the paper, set
$$\wS:=\kk[\, x_{i,j} \mid 1 \leq i \leq n, 1 \leq j \leq a_i+1 \,].$$
Note that $\dPol(x^\bb)=\prod_{i \in [n]} x_{i,a_i+1} x_{i,a_i} \cdots x_{i, a_i-b_i+2} \in \wS$.  
For a positively $\ba$-determined monomial ideal $I \subset S$, 
consider the squarefree monomial ideal 
$$\BM_{\ba}(I):= \Pol(I)+\dPol(I^{\vee \ba})+( \, \prod_{m=1}^{a_l+1}x_{l,m} \mid l=1,\ldots,n \, )$$
of $\wS$. Let $\BBM_\ba(I)$ be the simplicial complex whose Stanley-Reisner ring is $S/\BM_{\ba}(I)$. 
Murai \cite{Mu} showed that the geometric realization of $\BBM_\ba(I)$ is homeomorphic to a sphere of dimension  $|\ba|-2$. 

\begin{prop}\label{BM}
For all $i \in [n]$ and an integer $j$ with $1 \leq j \leq a_i+1$, we have 
$$\BM_{\ba+\be_i}(I^{\tl \< i,j \>})\cong 
\infl_{(i,j)} (\BM_\ba(I)),$$
where ``$\cong$" means the equivalence via suitable identification of the variables 
of the two polynomial rings (note that the each side of the above equation are monomial 
ideals of polynomial rings of $|\ba|+n+1$ variables, but the rings are not the same) and 
$\infl_{(i,j)}: \Sq \wS \to \Sq (\wS[x_{i,j}'])$ denotes the inflation at the variable $x_{i,j}$. 
\end{prop}

\begin{proof}
Note that $\BM_{\ba+\be_i}(I^{\tl \<i,j\>})$ is an ideal of $\hS:=\wS[x_{i,a_i+2}]$.  
Consider the injective ring homomorphism $\phi: \wS \to \hS$ defined by 
$$
\phi(x_{l,m})=\begin{cases}
x_{i,m +1} & \text{if $l=i$ and $m \geq j$,}\\
x_{l,m} & \text{otherwise.}\\
\end{cases}
$$
For a squarefree monomial $x^\bb \in \wS$, set 
$$
\psi(x^\bb):=\begin{cases}
x_{i, j} \cdot \phi (x^\bb) & \text{if $x_{i,j}$ divides $x^\bb$,}\\
\phi(x^\bb) & \text{otherwise.}
\end{cases}
$$
Note that $\psi$ is not a ring homomorphism, just a map
from the set of squarefree monomials of $\wS$ to that of $\hS$.
Clearly, we have 
$$\opn{\mathsf{pol}}_{\ba+\one+\be_i}(I^{\tl \<i,j\>}) = 
( \, \psi(x^\bb) \mid \text{$x^\bb$ is a minimal generator of $\Pol(I)$} \, )$$
and 
$$\prod_{m=1}^{a_l+1}x_{l,m} = \psi(\prod_{m=1}^{a_l}x_{l,m}).$$

By Lemma~\ref{Alex dual},  
$(I^{\tl \<i,j\>})^{\vee \ba+\be_i}= (x^{\tau_{\<i, a_i +2 -j\>}(\bb)} \mid 
\text{$x^\bb$ is a minimal generator of $I^{\vee \ba}$})$.  
On the other hand, for a monomial $\bb \in \NN$ with $\bb \preceq \ba$, it is easy to check that 
$$\opn{\mathsf{pol}}^{\ba+\one+\be_i}(x^{\tau_{\<i, a_i+2-j\>}(\bb)})
 = \psi(\dPol(x^\bb)).$$
Hence we have 
$$\opn{\mathsf{pol}}^{\ba+\one+\be_i}
((I^{\tl \<i,j\>})^{\vee \ba+\be_i}) = 
( \, \psi(x^\bb) \mid \text{$x^\bb$ is a minimal generator of $\dPol(I^{\vee \ba})$} \, ).$$
Combining the above facts, we get 
$$\BM_{\ba+\be_i}(I^{\tl \<i,j\>})=
(\, \psi(x^\bb) \mid \text{$x^\bb$ is a minimal generator of 
$\BM_{\ba}(I)$}\,).$$

On the other hand, $\psi$ induces the inflation at $x_{i,j}$.  
So we are done. 
\end{proof}

By a similar argument as Corollary~\ref{consecutive}, we can prove the following. 
This time, we can not ignore the case there is a variable $x_i$ which divides none 
of the minimal generators of a monomial ideal.    

\begin{cor}
Modulo 1-vertex inflations, any Bier-Murai sphere is obtained from $\BBM_\ba(I)$ of 
a monomial ideal $I \subset S$ satisfying the following conditions: 
\begin{itemize}
\item[(i)]  Let $x^{\ba_1}, \ldots, x^{\ba_r}$ be the minimal generators of $I$.  
If a variable $x_i$ divides some of $x^{\ba_1}, \ldots, x^{\ba_r}$, 
there is a positive integer $b_i$ such that
$\{\, (\ba_l)_i \mid 1 \le l \le r \, \}$ is $\{ \, 1, 2, \ldots, b_i \, \}$  
or $\{ \, 0,1, \ldots, b_i \, \}$ 
\item[(ii)] $\ba=\ba_1 \vee \ldots \vee \ba_r$.  
\end{itemize}
\end{cor}

\begin{exmp}
Consider the squarefree monomial ideal $I=(xyz, xw, yw)$ of $\kk[x,y,z,w]$. 
Then $I^{\vee\one} = (xy, xw, yw, zw)$, and $\BM_\one(I)$ is the ideal 
$$(x_1y_1z_1,x_1w_1,y_1w_1,x_2y_2, x_2w_2,y_2w_2, z_2w_2,x_1x_2,y_1y_2,z_1z_2,w_1w_2)$$
of $\kk[x_1,x_2,y_1,y_2,z_1,z_2,w_1,w_2]$. 
Set $I' := I^{\tl \<1,1\>}= (x^2yz, x^2w, yw)$. 
Then easy calculation shows that $I^{\vee \one+\be_1}=(xy,xw,yw, zw) \, (=I^{\vee \one})$. 
Hence $\BM_{\one+\be_1}(I')$ is the ideal 
$$(x_1x_2y_1z_1,x_1x_2w_1,y_1w_1,x_3y_2, x_3w_2,y_2w_2, z_2w_2,x_1x_2x_3,y_1y_2,z_1z_2,w_1w_2),$$
which is clearly equivalent to $\infl_{(1,1)}(\BM_\one(I))$. 
Next consider $\infl_{(1,2)}:\mod_{\ZZ^n}S \to \mod_{\ZZ^n}S$. While $\infl_{(1,2)}(I)=I=(xyz, xw, yw)$, 
its Alexander dual with respect to $\one+\be_1$ is $(x^2y, x^2w, yw,zw)$. Hence  
$\BM_{\one+\be_1}(\infl_{1,2}(I))$ is the ideal
$$(x_1y_1z_1,x_1w_1,y_1w_1,x_2x_3y_2, x_2x_3w_2,y_2w_2, z_2w_2,x_1x_2x_3,y_1y_2,z_1z_2,w_1w_2),$$
which coincides with $\infl_{(1,2)}(\BM_\one(I))$. 

In general, the following holds: Let $I \subset S$ be a positively $\ba$-determined ideal.   
Through the variable exchange $x_{i,j} \longleftrightarrow x_{i,a_i+1-j}$, 
$\BM_\ba(I)$ is equivalent to $\BM_\ba(I^{\vee \ba})$. In this context, 
$\infl_{(i,j)}(\BM_\ba(I))$ corresponds to $\infl_{(i,a_i+2-j)}(\BM_\ba(I^{\vee \ba}))$.

\end{exmp}


\begin{thebibliography}{99}
\bibitem{A1} J. Apel,
On a conjecture of R. P. Stanley; Part I -- Monomial ideals, 
J. Algebraic Combin. {\bf 17} (2003), 39--56.

\bibitem{A2} J. Apel,
On a conjecture of R. P. Stanley; Part II -- Quotients modulo monomial ideals,
J. Algebraic Combin. {\bf 17} (2003), 57 -- 74.

\bibitem{BF} M. Brun and Fl\o ystad, 
The Auslander-Reiten translate on monomial quotient rings, preprint. 


\bibitem{BH} W. Bruns and J. Herzog, On multigraded resolutions, 
Math. Proc. Cambridge Philos. Soc. {\bf 118} (1995) 245--257.

\bibitem{FT}A. Fr\"{u}bis-Kruger and  N. Terai,  
Bounds for the regularity of monomial ideals, 
Mathematiche (Catania) {\bf 53} (Suppl.) (1998) 83--97.

\bibitem{HVZ} J. Herzog, M. Vladoiu, X. Zheng,
How to compute the Stanley depth of a monomial ideal,
J. Algebra {\bf 322} (2009), 3151--3169. 

\bibitem{M} E. Miller,
The Alexander duality functors and local duality with monomial support,
J. Algebra {\bf 231} (2000), 180--234.

\bibitem{MS} E. Miller and B. Sturmfels,
Combinatorial Commutative Algebra, Graduate Texts in Mathematics {\bf 227},
Springer, 2005.

\bibitem{Mu} S. Murai, Spheres arising from multicomplexes, preprint (arXiv:1002.1211). 

\bibitem{OY} R. Okazaki and K. Yanagawa, 
Alexander duality and Stanley decomposition of multi-graded modules, preprint, 2010 
({\tt arXiv:1003.4008}). 

\bibitem{Sb} 
E. Sbarra, Upper bounds for local cohomology for rings with given Hilbert function.
Comm. Algebra {\bf 29} (2001), 5383--5409.

\bibitem{St}R. Stanley, 
``Combinatorics and commutative algebra," 2nd ed., Birkh\"auser, 1996. 



\bibitem{V} W. V. Vasconcelos, {\it Cohomological degrees of graded modules}, In: gSix lectures on commutative
algebrah (J. Elias, J. M. Giral, R. M. Miro-Roig, S. Zarzuela, Eds.), Progress in
Mathematics 166, Birkh\"{a}user, Basel, 1998, pp. 345--392. 


\bibitem{Y} K. Yanagawa,
Alexander duality for Stanley-Reisner rings and squarefree $\NN^n$-graded modules,
J. Algebra {\bf 225} (2000), 630--645.
\end{thebibliography}
\end{document}